\documentclass{siamart1116}
\usepackage{graphicx}
\usepackage{amssymb,amsmath}
\usepackage{algorithm}
\usepackage{algorithmic}
\usepackage{booktabs}
\usepackage{multirow}
\usepackage{balance}
\usepackage{url}

\newcommand{\mbfY}{\mathbf{Y}}
\newcommand{\mbfZ}{\mathbf{Z}}
\newcommand{\mbfU}{\mathbf{U}}
\newcommand{\mbfV}{\mathbf{V}}

\newcommand{\mbfQ}{\mathbf{Q}}
\newcommand{\mbfP}{\mathbf{P}}

\newcommand{\mbfA}{\mathbf{A}}

\usepackage{subfigure}
\usepackage{hyperref}

\newtheorem{assumption}{Assumption}

\usepackage{multicol}
\usepackage{multirow}
\catcode`~=11 \def\UrlSpecials{\do\~{\kern -.15em\lower .7ex\hbox{~}\kern .04em}} \catcode`~=13 

\DeclareMathAlphabet{\mathbsf}{OT1}{cmss}{bx}{n}
\DeclareMathAlphabet{\mathssf}{OT1}{cmss}{m}{sl}

\DeclareSymbolFont{bsfletters}{OT1}{cmss}{bx}{n}  
\DeclareSymbolFont{ssfletters}{OT1}{cmss}{m}{n}
\DeclareMathSymbol{\bsfGamma}{0}{bsfletters}{'000}
\DeclareMathSymbol{\ssfGamma}{0}{ssfletters}{'000}
\DeclareMathSymbol{\bsfDelta}{0}{bsfletters}{'001}
\DeclareMathSymbol{\ssfDelta}{0}{ssfletters}{'001}
\DeclareMathSymbol{\bsfTheta}{0}{bsfletters}{'002}
\DeclareMathSymbol{\ssfTheta}{0}{ssfletters}{'002}
\DeclareMathSymbol{\bsfLambda}{0}{bsfletters}{'003}
\DeclareMathSymbol{\ssfLambda}{0}{ssfletters}{'003}
\DeclareMathSymbol{\bsfXi}{0}{bsfletters}{'004}
\DeclareMathSymbol{\ssfXi}{0}{ssfletters}{'004}
\DeclareMathSymbol{\bsfPi}{0}{bsfletters}{'005}
\DeclareMathSymbol{\ssfPi}{0}{ssfletters}{'005}
\DeclareMathSymbol{\bsfSigma}{0}{bsfletters}{'006}
\DeclareMathSymbol{\ssfSigma}{0}{ssfletters}{'006}
\DeclareMathSymbol{\bsfUpsilon}{0}{bsfletters}{'007}
\DeclareMathSymbol{\ssfUpsilon}{0}{ssfletters}{'007}
\DeclareMathSymbol{\bsfPhi}{0}{bsfletters}{'010}
\DeclareMathSymbol{\ssfPhi}{0}{ssfletters}{'010}
\DeclareMathSymbol{\bsfPsi}{0}{bsfletters}{'011}
\DeclareMathSymbol{\ssfPsi}{0}{ssfletters}{'011}
\DeclareMathSymbol{\bsfOmega}{0}{bsfletters}{'012}
\DeclareMathSymbol{\ssfOmega}{0}{ssfletters}{'012}

\DeclareMathOperator*{\dist}{dist}

\usepackage{amsfonts}
\usepackage{graphicx}
\usepackage{epstopdf}
\usepackage{algorithmic}
\ifpdf
  \DeclareGraphicsExtensions{.eps,.pdf,.png,.jpg}
\else
  \DeclareGraphicsExtensions{.eps}
\fi

\numberwithin{theorem}{section}

\newcommand{\TheTitle}
{DCA based Algorithm with Extrapolation for Nonconvex Nonsmooth Optimization}

\newcommand{\ShortTitle}{%
DCA based Algorithm with Extrapolation} 

\newcommand{\TheAuthors}{D. N. Phan, H. A. Le Thi}
\headers{\ShortTitle}{\TheAuthors}

\title{{\TheTitle}}
\author{
   Duy Nhat Phan\thanks{Department of Mathematics and Informatics, HCMC University of Education, Vietnam (\email{nhatpd@hcmue.edu.vn}).}
  \and
  Hoai An Le Thi\thanks{Department Computer Science and Application, LGIPM, University of Lorraine, France
             (\email{hoai-an.le-thi@univ-lorraine.fr}).}
 }

\usepackage{amsopn}

\ifpdf
\hypersetup{
  pdftitle={\TheTitle},
  pdfauthor={\TheAuthors}
}
\fi

\definecolor{brightpink}{rgb}{1.0, 0.0, 0.5}

\allowdisplaybreaks
\begin{document}
\maketitle
\begin{abstract}
In this paper, we focus on the problem of minimizing the sum of a nonconvex differentiable function and a DC (Difference of Convex functions) function, where the differentiable function is not restricted to the global Lipschitz gradient continuity assumption. This problem covers a broad range of applications in machine learning and statistics such as compressed sensing, signal recovery, sparse dictionary learning, and matrix factorization etc. We take inspiration from the Nesterov's acceleration technique and the DC algorithm to develop a novel algorithm for the considered problem. Analyzing the convergence, we study the subsequential convergence of our algorithm to a critical point. Furthermore, we justify the global convergence of the whole sequence generated by our algorithm to a critical point, and establish its convergence rate under the Kurdyka-\L{}ojasiewicz condition. Numerical experiments on the nonnegative matrix completion problem are performed to demonstrate the efficiency of our algorithm and its superiority over well-known methods.
\end{abstract}
\section{Introduction}

In this paper, we specifically focus on a class of following nonconvex nonsmooth problems
\begin{equation}\label{model}
\min_{x\in\mathbb E} F(x) := f(x) + g(x) - h(x),
\end{equation}
where $\mathbb E$ is a finite dimensional real linear space and the functions $f, g$, and $h$ satisfy the following mild assumption:
\begin{assumption}\label{assump1}
\begin{itemize}
    \item[(a)] $f:\mathbb E \to \mathbb{R}$ is a nonconvex continuously differentiable function and there exists a strongly convex continuously differentiable function $\phi$ such that $L\phi - f$ and $l\phi + f$ are convex for some $L>0$ and $l\geq 0$. 
    
    \item[(b)] $g:\mathbb E \to \mathbb{R}\cap \{+\infty\}$ is a proper and lower semicontinuous convex function.
    
    \item[(c)] $h:\mathbb E\to \mathbb{R}$ is a convex function.
    \item[(d)] $F$ is bounded from below.
\end{itemize}
\end{assumption}
Problem \eqref{model} covers a broad range of applications in machine
learning and statistics, e.g., compressed sensing \cite{doncom}, signal recovery \cite{becspa}, sparse dictionary learning \cite{ahaksv}, sparse matrix factorization \cite{bolpro}, quadratic inverse problem in phase retrieve \cite{bolfir}, and Poisson linear inverse problems \cite{berima}. In particular, many optimization problems in these applications can be formulated as a regularized problem of the form
\begin{equation}\label{regularizedp}
    \min_{x\in\mathbb E}\biggl\{ f(x) + r(x)\biggr\},
\end{equation}
where $f$ is a loss function and $r$ is a regularization term for promoting sparse solutions.  We notice that if $L$-smooth and $L$-smooth adaptable \cite{bolfir} loss functions $f$ satisfy Assumption \ref{assump1} (a). Hence, the problem \eqref{regularizedp} can be recast into Problem \eqref{model} for some well-known penalty functions $r$ which are DC functions (Difference of Convex functions) such as SCAD \cite{Fan-2001}, MCP \cite{Zhang-2010}, and Exponential penalty function \cite{Bradley1998}.
In this paper, we are interested in the nonnegative matrix completion problem corresponding to \eqref{model}.
Given a nonnegative sparse matrix $\mbfA\in\mathbb{R}^{m\times n}$, in order to obtain factors $\mbfU\in\mathbb{R}^{m\times t}_+$ and $\mbfV\in\mathbb{R}^{t\times n}_+$ such that $\mbfA \approx \mbfU\mbfV$, we need to solve the following nonconvex optimization problem
\begin{equation}\label{MF}
    \min_{\mbfU\in\mathbb{R}^{m\times t}_+,\mbfV\in\mathbb{R}^{t\times n}_+} \biggl\{\frac{1}{2}\|\mathcal P(\mbfA - \mbfU\mbfV)\|_F^2 + r(\mbfU,\mbfV)\biggr\},
\end{equation}
where $f(\mbfU,\mbfV): = \frac{1}{2}\|\mathcal P(\mbfA - \mbfU\mbfV)\|_F^2$ is the data-fitting term, $r$ is a regularization term. Here, $\mathcal P(\mbfZ)_{ij} = \mbfZ_{ij}$ if $\mbfA_{ij}$ is observed and $0$ otherwise. The matrix completion problem taking the form of \eqref{MF} is one of the most successful approaches in recommendation system, see, e.g. \cite{kormat}. The usage of matrix completion can also be found in a wide range of tasks from sensor networks \cite{Biswas2006}, social network analysis \cite{Kim2011} to image processing \cite{Lui2013}.

\subsection{Related works}
The considered problem \eqref{model} can be expressed as a DC program
\begin{equation}\label{DCp}
    \min_{x\in\mathbb E}\biggr\{F(x) : = G(x) - H(x)\biggl\},
\end{equation}
where $G(x) = L\phi(x) + g(x)$ and $H(x) = L\phi(x) - f(x) + h(x)$ are convex functions. Hence, the powerful classical DCA \cite{letthe,phacon} for handling \eqref{DCp} iteratively linearizes $H$ by computing $v^k\in \partial H(x^k)$ and computes the next iterate $x^{k+1}$ by solving the convex sub-problem
\begin{equation}\label{sub-DC}
    \min_{x\in\mathbb E}\biggr\{G(x) - \langle v^k, x \rangle\biggl\}.
\end{equation}
\cite{aniwel} proposed an inertial DCA (iDCA) by incorporating a weak inertial force of the heavy-ball method of \cite{sompol} into DCA. More precisely, iDCA replaces $v^k$ in the convex subproblem \eqref{sub-DC} by $v^k + \gamma(x^k - x^{k-1})$ with $\gamma \geq 0$. This leads to the following convex sub-problem
\begin{equation*}
    \min_{x\in\mathbb E}\biggr\{G(x) - \langle v^k + \gamma(x^k - x^{k-1}), x \rangle\biggl\}.
\end{equation*}
In order to study the convergence of iDCA, \cite{aniwel} assumed that the second DC component $H$ is strongly convex, but this assumption can be replaced by the strongly convexity of the first DC component $G$.
In \cite{Phan2018_IJCAI}, by taking inspiration from the Nesterov's acceleration technique \cite{Nesterov1983}, the authors introduced an accelerated DCA (ADCA) that constructs an extrapolation point $y^k$ of the current iterate $x^k$ and the previous one $x^{k-1}$
\begin{equation*}
y^k = x^k +  \frac{t_k-1}{t_{k+1}}\biggl(x^k-x^{k-1}\biggr),
\end{equation*}
where $t_{k+1} = (1+\sqrt{1+4t^2_{k}})/2$. ADCA then verifies the following condition
\begin{equation}\label{equ:cond}
    F(y^k)\leq \max_{t = [k- q]_+,...,k}F(x^t),
\end{equation}
where $[k- q]_+ = \max(0,k- q)$. If the condition \eqref{equ:cond} is satisfied, $y^k$ will be used instead of $x^k$ to linearize $H$ by computing $v^k\in H(y^k)$. 
Like DCA, ADCA computes $x^{k+1}$ by solving the subproblem  \eqref{sub-DC}. However, the computing of the objective function values $F(y^k)$ may be expensive in some cases or the condition \eqref{equ:cond} is less likely to be satisfied in constrained problems.

For a special case of Problem \eqref{model} in which $f$ admits a $L$-Lipschitz continuous gradient, we can choose the function $\phi(x) = \frac{1}{2}\|x\|^2$. Hence, the sub-problem \eqref{sub-DC} of DCA applied to this particular case becomes the proximal one
\begin{equation*}
    \min_{x\in\mathbb E}\biggr\{\frac{L}{2}\|x-x^k\|^2  + \langle \nabla f(x^k) - \xi^k, x \rangle + g(x)\biggl\},
\end{equation*}
where $\xi^k\in\partial h(x^k)$. This DCA scheme is noting else but the proximal DCA algorithm (pDCA) proposed in \cite{Gotoh-2017} for the particular case.  In \cite{wenapr}, the authors proposed a version of pDCA with extrapolation (pDCAe) for solving a special setting of \eqref{model} in which $f$ is convex differentiable with $L$-Lipschitz continuous gradient. In particular, pDCAe first computes the extrapolation point $y^k = x^k + \beta_k(x^k-x^{k-1})$, where $\beta_k\in [0,\beta]$ with $\beta\in[0,1)$, and then computes the next iterate $x^{k+1}$ by solving the proximal sub-problem
\begin{equation*}
    \min_{x\in\mathbb E}\biggr\{\frac{L}{2}\|x-y^k\|^2  + \langle \nabla f(y^k) - \xi^k, x \rangle + g(x)\biggl\}.
\end{equation*}
For this case, \cite{zhaenh,zhanon} additionally assumed that $h$ is the supremum of finitely many convex smooth functions and proposed several variants of pDCAe for solving the resulting problem. 

Besides to DCA based algorithmms, popular proximal gradient (PG) algorithms and their accelerated versions have developed for special cases of \eqref{model} see, e.g. \cite{kappro,attcon,ngucon,Li2015,Yao2017}. However, these methods also require the global Lipschitz gradient continuity of $f$. In fact, this requirement may not often be satisfied for many practical problems, such as the quadratic inverse problem in phase retrieve \cite{bolfir}, Poisson linear inverse problems \cite{berima}, and the matrix factorization \eqref{MF}, that arise in many real world
applications. Recently, in \cite{bolfir}, the authors proposed a Bregman PG (BPG) algorithm for minimizing the sum of two nonconvex functions $f$ and $r$. \cite{mukcon,mukbey} developed an extrapolated version of BPG for the same problem. Although they do not impose the global Lipschitz gradient continuity on $f$, they require the weakly convexity of $r$. In addition, the aforementioned PG algorithms
have to compute the proximal map of nonconvex functions $r$
which do not has closed form in many cases. Usually, this
computation can be very expensive or impossible.

\subsection{Contribution}
Motivated by the advantages of the success of Nesterov's acceleration technique for convex programming to which it was accelerated with an $\mathcal{O}(1/k^2)$ convergence rate \cite{Nesterov1983,Beck2009}, we thereby aim to investigate a novel algorithm with extrapolation for solving the DC program \eqref{model}, named DCAe, by incorporating the Nesterov-type acceleration into DCA. In particular, the DCAe algorithm linearizes the first part $L\phi - f$ of $H$ at the extrapolated point $y^k$ and the second one $h$ at the current iterate $x^k$. DCAe subsumes pDCAe \cite{wenapr} as a special version. In addition, it is important noting that our problem \eqref{model} is more general than the problems considered in \cite{Gotoh-2017,wenapr,zhaenh,zhanon}. Since we do not require the global Lipschitz gradient continuity or convexity of $f$, the analysis in \cite{wenapr,zhaenh,zhanon} is not applicable to Problem \eqref{model}. Theoretically, we show that every limit point of the generated by DCAe is a critical point. Importantly, we study the global convergence of the whole sequence generated by DCAe, and establish its convergence rate with the Kurdyka-\L{}ojasiewicz (KL) assumption. As application, we take the nonnegative matrix completion problem into consideration. We perform numerical experiments on different datasets to demonstrate the efficiency of the proposed algorithm.

\subsection{Organization of the paper}
In section \ref{pre}, we present some basic concepts in optimization and Kurdyka-\L{}ojasiewicz property. We introduce the novel DCAe algorithm and investigate its convergence properties in section in section \ref{mainalgo}. Numerical experiments are showed in section \ref{numerical}. Finally, in section \ref{conclu}, we provide conclusions of the study.



\section{Preliminaries}\label{pre}
Before introducing the proposed algorithm, let us first recall some basis notations that will be used in the sequel.

Given a lower semicontinuous function $J: \mathbb E \to (-\infty, +\infty]$, its domain is defined by
\begin{equation*}
    dom(J): = \{x\in\mathbb E: J(x)<  +\infty\}.
\end{equation*}
The modulus of strong convexity of $J$ on $\Omega$, denoted by $\rho(J,\Omega)$ or $\rho(J)$ if $\Omega = \mathbb E$, is given by
\begin{equation}
    \rho(J,\Omega) = \sup \{\rho \ge 0 : J - (\rho /2)\|\cdot\|^2 \text{~is convex on} ~\Omega\}.
\end{equation}
One says that $J$ is $\rho$ - \emph{strongly convex} on $\Omega$ if $\rho(J,\Omega) > 0$.

For a given $x^*\in dom(J)$, the Fr\'echet subdifferential \cite{Mor06} of a lower semicontinous function $J$ at $x^*$ is defined by
\begin{equation*}
\partial ^FJ(x^*) = \left\{u:\liminf_{x\to x^*}\frac{J(x) - J(x^*) - \langle u,x - x^*\rangle}{\|x - x^*\|}\geq 0 \right\}.
\end{equation*}
The limiting subdiferential \cite{Mor06} of $J$ at $x^*\in dom(J)$ is defined by
\begin{equation*}
\partial ^LJ(x^*) = \left\{u:\exists x^k \xrightarrow{J} x^*, u^k\in \partial ^FJ(x^k), u^k\to u \right\},
\end{equation*}
where the notation $x^k \xrightarrow{J} x^*$ means that $x^k \to x^*$ and $J(x^k) \to J(x^*)$. If $J$ is differentiable at $x^*$, then $\partial^FJ(x^*) = \{\nabla J(x^*)\}$. Moreover, if $J$ is continuously differentiable on a neighborhood of $x^*$, then $\partial^LJ(x^*) = \{\nabla J(x^*)\}$. We note that if $J = \phi + \psi$, where $\phi$ is lower semicontinuous, and $\psi$ is continuously differentiable on a neighborhood of $x^*$, then
\begin{equation*}
\partial ^LJ(x^*) = \partial ^L\phi(x^*) + \nabla \psi(x^*).
\end{equation*}

If $J$ is convex, then the Fr\'echet and the limiting subdifferential reduce to the subdifferential in the sense of convex analysis:
\[
\partial J (x^*) := \{y : J(x) \ge J(x^*) + \langle x-x^*,y\rangle, \forall x \in \mathbb E\}.
\]

A point $x^*$ is called a critical point of the problem \eqref{model} if and only if $[\nabla f(x^*) + \partial g(x^*) \cap \partial h(x^*) \neq \emptyset$. 

Given $\eta\in(0,\infty]$, we denote by $\mathcal{M}_\eta$ the class of continuous concave functions $\psi:[0,\eta) \rightarrow [0,\infty)$ verifying 
\begin{itemize}
\item[(i)] $\psi(0) = 0$ and $\psi$ is continuously differentiable on $(0,\eta)$,
\item[(ii)] $\psi'(t)>0$ for all $t\in(0,\eta)$.
\end{itemize}
A lower semicontinuous function $J$ satisfies the KL property \cite{Attouch2010} at $u^*\in dom(\ \partial^L J)$ if there exists $\eta >0$, a neighborhood $\mathcal{V}$ of $u^*$, and $\psi\in \mathcal{M}_\eta$ such that for all $u\in\mathcal{V}\cap\{u:J(u^*)<J(u)<J(u^*)+\eta\}$, one has
\begin{equation*}
\psi'(J(u) - J(u^*))\dist(0,\partial^LJ(u)) \geq 1.
\end{equation*}
We notice that the class of functions verifying the KL property is very ample, for example, semi-algebraic, real analytic, and log-exp functions see, e.g. \cite{attcon,Attouch2010,hedon}.

\section{Difference of Convex function Algorithm with Extrapolation}\label{mainalgo}
\subsection{Algorithm description}
Let us introduce a novel algorithm to solve the nonconvex nonsmooth optimization problem \eqref{model}. In order to accelerate DCA, we propose its version with extrapolation, named DCAe that, at each iteration, constructs an extrapolation point $y^k$ of the current iterate $x^k$ and the previous one $x^{k-1}$
\begin{equation*}
    y^k = x^k + \beta_k(x^k - x^{k-1}),
\end{equation*}
where $\beta_k$ is an extrapolation parameter satisfying
\begin{equation*}
 (L+l)D_\phi(x^k, y^k) \leq  \delta L D_\phi(x^{k-1}, x^k),
\end{equation*}
with $\delta\in (0,1)$ and $D_\phi(x,y):= \phi(x) - [\phi(y) + \langle\nabla \phi(y), x - y\rangle]$.
The DCAe employs both $y^k$ and $x^k$ in order to linearize $H$. In particular, the first part $L\phi - f$ of $H$ is linearized at $y^k$ by computing $\nabla L\phi(y^k) - \nabla f(y^k)$ while the second one $h$ is linearized at $x^k$ by computing $\xi^k\in\partial h(x^k)$. The next iterate $x^{k+1}$ is then computed by solving the following convex sub-problem
\begin{equation*}
    \min_{x\in\mathbb E}\biggr\{G(x) - \biggl\langle \nabla L\phi(y^k) - \nabla f(y^k) + \xi^k, x \biggr\rangle\biggl\}.
\end{equation*}
The DCAe algorithm is summarized in Algorithm \ref{DCAe}.
\begin{algorithm}[H]
\caption{DCAe for solving \eqref{model}}
\begin{algorithmic}\label{DCAe}
   \STATE {\bfseries Initialization:} Choose $x^0 = x^{-1}$, $\delta\in (0,1)$, $L>0$, $l\geq 0$, and $k\leftarrow 0$. 
   \REPEAT 
   \STATE 1. Compute $y^k = x^k + \beta_k(x^k - x^{k-1})$, where $\beta_k$ satisfies
\begin{equation}\label{equ:beta}
 (L+l)D_\phi(x^k, y^k) \leq  \delta L D_\phi(x^{k-1}, x^k)
\end{equation}
    \STATE 2. Compute $x^{k+1}$ by solving the convex sub-problem
    \begin{equation}
    \min_{x\in\mathbb E}\biggr\{G(x) - \biggl\langle L\phi(y^k) - \nabla f(y^k) + \xi^k, x \biggr\rangle\biggl\},
\end{equation}
where $\xi^k\in\partial h(x^k)$.
   \STATE 3. $k\leftarrow k+ 1$.
    \UNTIL{Stopping criterion.}
\end{algorithmic}
\end{algorithm}
Now, let us provide some special choices of the extrapolation parameter. We first consider the special case in which $f$ admits the $L$-Lipschitz continuous gradient. Hence, we can choose the function $\phi(x) = \frac{1}{2}\|x\|^2$. Condition \eqref{equ:beta} becomes
\begin{equation*}
    (L+l)\beta_k^2\|x^k-x^{k-1}\|^2 \leq \delta L\|x^k-x^{k-1}\|^2.
\end{equation*}
Therefore, we choose $\beta_k$ such that $\beta_k \leq \sqrt{\frac{\delta L}{L+l}}$.
Moreover, if $f$ is convex, we can take $l=0$ and $\beta_k\leq \sqrt{\delta}$. 
In general case, \cite[Lemma 4.1]{mukcon} showed that there always exists $\gamma_k>0$ such that the condition \eqref{equ:beta} is satisfies for all $\beta_k\in[0,\gamma_k]$. Therefore, $\beta_k$ can be determined by a line search as follows. At each iteration, we initialize $\beta_k = \frac{\mu_k-1}{\mu_k}$, and while the inequality \eqref{equ:beta} does not hold, we decrease $\beta_k$ by $\beta_k = \beta_k\eta$ with $\eta \in(0,1)$, where $\mu_k = \frac{1}{2}(1+\sqrt{1+4\mu_{k-1}^2})$ and $\mu_0=1$.

\subsection{Convergence Analysis}
\subsubsection{Sub-sequential convergence}
In this subsection, we study the convergence of the DCAe. The first result presents the sub-sequential convergence of the sequence $\{x^k\}$ generated by DCAe. 
\begin{theorem}\label{theo1}
Let $\{x^k\}$ be the sequence generated by DCAe. Suppose Assumption \ref{assump1} is satisfied.  The following statements are hold.

\begin{itemize}
    \item[a)] For $k=0,1,...$, we have
    \begin{equation}\label{propertya}
    F(x^{k+1}) \leq F(x^k) +\delta L D_\phi(x^{k-1},x^k) \ - L D_\phi(x^k,x^{k+1}).
\end{equation}
    \item[b)] $\sum\limits_{k=0}^{+\infty}\|x^k - x^{k-1}\|^2 < +\infty$, and thus $\lim\limits_{k\to +\infty}\|x^k - x^{k-1}\| = 0$.
    \item[c)] If $x^*$  is a limit point of $\{x^k\}$, then $x^*$ is a critical point of \eqref{model}.
\end{itemize}
\end{theorem}

\begin{proof}
a) By the definition of the gradient $\xi^k\in\partial h(x^k)$, we have
\begin{equation*}
    h(x^{k+1}) \geq h(x^k) + \langle \xi^k, x^{k+1}- x^k\rangle ,
\end{equation*}
which is equivalent to
\begin{equation}
    \begin{split}
        -h(x^{k+1}) \leq -h(x^k) - \langle \xi^k, x^{k+1}- x^k\rangle 
    \end{split}
\end{equation}
Similarly, we have
\begin{equation*}
\begin{split}
    L\phi(x^{k+1}) - f(x^{k+1})\geq L\phi(y^k) - f(y^k)
    + \langle L\nabla\phi(y^k) - \nabla f(y^k), x^{k+1} - y^k\rangle,
\end{split}
\end{equation*}
which is equivalent to 
\begin{equation}
    \begin{split}
        f(x^{k+1}) \leq f(y^k) + L(\phi(x^{k+1}) - \phi(y^k))
          - \langle L\nabla\phi(y^k) - \nabla f(y^k), x^{k+1} - y^k\rangle.
    \end{split}
\end{equation}
Since $x^{k+1}$ is a solution to the convex sub-problem \eqref{sub-DC}, we have
\begin{equation*}
    L\nabla\phi(y^k) - \nabla f(y^k) + \xi^k - L\nabla\phi(x^{k+1}) \in\partial g(x^{k+1}).
\end{equation*}
This implies that
\begin{equation}
    \begin{split}
         g(x^{k+1}) \leq  g(x^k) - \langle L\nabla\phi(y^k) + \xi^k - \nabla f(y^k)
        - L\nabla\phi(x^{k+1}), x^k - x^{k+1}\rangle
    \end{split}
\end{equation}
Summing the three inequalities above, we get
\begin{equation}\label{eq:10111}
    \begin{split}
        F(x^{k+1}) \leq g(x^k) - h(x^k) + f(y^k) + L\phi(x^{k+1}) 
         - L\phi(y^k)\\
         - \biggl\langle L\nabla\phi(y^k) -  \nabla f(y^k), x^k - y^k\biggr\rangle 
        - \biggl\langle L\nabla\phi(x^{k+1}), x^{k+1} - x^k\biggr\rangle\\
        = g(x^k) - h(x^k) + f(y^k) + \biggl\langle \nabla f(y^k), x^k - y^k\biggr\rangle 
         + L D_\phi(x^k,y^k) - L D\phi(x^k,x^{k+1}).
    \end{split}
\end{equation}
On the other hand, by the convexity of $l\phi + f$, we have
\begin{equation}
\begin{split}
    l\phi(x^k) + f(x^k) \geq l\phi(y^k) + f(y^k)  
     +  \biggl\langle l\nabla \phi(y^k) + \nabla f(y^k), x^k - y^k\biggr\rangle
\end{split}
\end{equation}
which is equivalent to
\begin{equation}\label{eq:10112}
    \begin{split}
       f(y^k) +  \biggl\langle \nabla f(y^k), x^k - y^k\biggr\rangle  \leq f(x^k) + lD_\phi(x^k,y^k)
    \end{split}
\end{equation}
Summing the two inequalities \eqref{eq:10111} and \eqref{eq:10112} above gives us that
\begin{equation}
\begin{split}
    F(x^{k+1}) \leq F(x^k) + (L+ l)D_\phi(x^k,y^k) 
     - L D\phi(x^k,x^{k+1})\\        
          \leq F(x^k) + \delta L D_\phi(x^{k-1},x^k) \ - L D\phi(x^k,x^{k+1}),
\end{split}
\end{equation}
where the second inequality follows from the definition of the parameter $\beta_k$. 

b) 
Summing the inequality \eqref{propertya} over $k=0,...,N$ gives that
\begin{equation*}
    \begin{split}
        \sum_{k=1}^N(1-\delta)LD_\phi(x^{k-1},x^k) \leq  F(x^0) + LD_\phi(x^{-1},x^0) - F(x^{N+1})
        - LD_\phi(x^N,x^{N+1}) \leq F(x^0) - \alpha,
    \end{split}
\end{equation*}
where $\alpha = \inf_{x\in \mathbb E}F(x)$. This inequality and the $\rho$-strongly convexity of $\phi$ imply that
\begin{equation}
    \sum_{k=0}^N(1-\delta)L\rho\|x^k - x^{k-1}\|^2\leq F(x^0) - \alpha
\end{equation}
Therefore, we have
\begin{equation}
    \sum_{k=0}^N\|x^k - x^{k-1}\|^2\leq \frac{ F(x^0) - \alpha}{(1-\delta)L\rho}.
\end{equation}
By taking to the limit, we have
\begin{equation}
    \sum_{k=0}^{+\infty}\|x^k - x^{k-1}\|^2 < +\infty,
\end{equation}
and $\lim_{k\to+\infty}\|x^k - x^{k-1}\| = 0$.

c) Let $\{x^{k_j}\}$ be a subsequence of $\{x^k\}$ that converges to $x^*$. It follows from (b) and $0\leq \beta_k<1$ that $\lim x^{k_j+1} = \lim x^{k_j-1} = \lim y^{k_j} = x^*$. Without loss generality, we can assume that $\xi^{k_j}$ converges to $\xi^*$. By the closedness of $\partial h(x^{k_j})$, we get that $\xi^*\in\partial h(x^*)$. From the definition of $x^{k_j+1}$, we have
\begin{equation}
    \begin{split}
        L\nabla\phi(y^{k_j}) - \nabla f(y^{k_j}) + \xi^{k_j} 
        - L\nabla\phi(x^{k_j+1}) \in\partial g(x^{k_j+1}).
    \end{split}
\end{equation}
Thanks to the continuity of $\nabla\phi$ and $\nabla f$, and the closedness of $\partial g$, passing to the limit gives us that
\begin{equation}
    - \nabla f(x^*) + \xi^* \in\partial g(x^*).
\end{equation}
Therefore, we obtain that
\begin{equation}
    \xi^* \in [\nabla f(x^*) + \partial g(x^*)]\cap \partial h(x^*).
\end{equation}
This completes the proof.
\end{proof}

\subsubsection{Global convergence}

Denote $\Phi$ be an auxiliary function by
\begin{equation}
    \Phi(x,y) = F(x) + \frac{(1+\delta)L}{2} D_\phi(y,x),
\end{equation}
The following theorem shows the global convergence of the whole bounded sequence generated by DCAe under an additional assumption below.
\begin{assumption}\label{assump2}
\begin{itemize}
    \item[(a)] $\Phi$ satisfies the KL property at any point $x \in dom\partial^L\Phi$.
    \item[(b)] $\nabla f$ and $\nabla \phi$ are Lipschitz continuous on any bounded subset of $\mathbb E$.
    \item[(c)] $h$ is differentiable with $\nabla h$ being Lipschitz continuous on any bounded subset of $\mathbb E$.
\end{itemize}
\end{assumption}

\begin{theorem}\label{theo2}
Let $\{x^k\}$ be a sequence generated by DCAe. Suppose that Assumptions \ref{assump1} and \ref{assump2}  are satisfied. If $\{x^k\}$ is bounded, it converges to a critical point of \eqref{model}. 
\end{theorem}
\begin{proof}
Denote $z^k = (x^k,x^{k-1})$. We first prove that the sequence $\{z^k\}$ satisfy three conditions $\mathbf H1$, $\mathbf H2$, and $\mathbf H3$ in \cite{attcon}.

($\mathbf H1$) \emph{Sufficient decrease condition}. From the inequality \eqref{propertya} and the $\rho$-strongly convexity of $\phi$, we derive that for all $k$
    \begin{equation*}
          \Phi(z^k) - \Phi(z^{k+1}) \geq \frac{(1-\delta)L\rho}{4}\|z^k - z^{k+1}\|^2.
    \end{equation*}
    
($\mathbf H2$) \emph{Relative error condition}. We need to justify that there exists $\kappa$ such that for all $k$
    \begin{equation*}
        \dist\biggl(0,\partial^L\Phi(z^{k+1})\biggr) \leq \kappa \|z^k - z^{k+1}\|.
    \end{equation*}
By the boundedness of $\{x^k\}$ and the locally Lipschitz continuity of $\nabla h, \nabla \phi, \nabla f$, there exists $L_h,L_\phi, L_f>0$ such that for all $k$
\begin{equation}\label{equ8141}
    \begin{split}
        \|\nabla h(x^k) - \nabla h(x^{k+1})\| &\leq L_h\|x^{k+1}-x^k\|\\
        \|\nabla \phi(y^k) - \nabla \phi(x^{k+1})\| &\leq L_\phi\|y^k-x^{k+1}\|\\
        \|\nabla f(y^k) - \nabla f(x^{k+1})\| &\leq L_\phi\|y^k-x^{k+1}\|.
    \end{split}
\end{equation}
It follows from the definition of $x^{k+1}$ that
\begin{equation}
    \begin{split}
        L\phi(y^k) - \nabla f(y^k) + \nabla h(x^k) - L\phi(x^{k+1}) \in \partial g(x^{k+1}).
    \end{split}
\end{equation}
Therefore, we have
\begin{equation}\label{eq:10113}
    \begin{split}
        L\biggl(\phi(y^k) - \phi(x^{k+1})\biggr) + \nabla f(x^{k+1}) - \nabla f(y^k) 
         + \nabla h(x^k) - \nabla h(x^{k+1})
         \in \partial^L F(x^{k+1})
    \end{split}
\end{equation}
On the other hand, we have
\begin{equation}\label{eq:10114}
    \partial^L\Phi(z^{k+1}) = \begin{pmatrix}
         \partial^L F(x^{k+1}) + (1+\delta)L/2\langle \nabla^2 \phi(x^{k+1}), x^{k+1} - x^k\rangle \\
           (1+\delta)L/2\biggl(\phi(x^k) - \phi(x^{k+1})\biggr)
         \end{pmatrix}.
\end{equation}
From this, \eqref{eq:10113}, and \eqref{equ8141}, we have
\begin{equation*}
    \begin{split}
        \dist\biggl(0,\partial^L\Phi(z^{k+1})\biggr) \leq \|L(\phi(y^k) - \phi(x^{k+1})) + \nabla f(x^{k+1})
        - \nabla f(y^k) + \nabla h(x^k) - \nabla h(x^{k+1}) \\ + (1+\delta)L/2\langle \nabla^2 \phi(x^{k+1}), x^{k+1} - x^k\rangle \|
        + \|(1+\delta)L/2(\phi(x^k) - \phi(x^{k+1}))\|\\
         \leq (L L_\phi + L_f)\|x^{k+1} - y^k\|
        + (L_h+(1+\delta)L/2(M + L_\phi))\|x^{k+1}-x^k\|\\
        \leq (L L_\phi + L_f)\|x^k - x^{k-1}\| + (LL_\phi + L_f+ L_h+(1+\delta)L/2(M + L_\phi))
        \|x^{k+1}-x^k\|,
    \end{split}
\end{equation*}
where we have used the facts that $\|\nabla^2 \phi(x^{k+1})\|\leq M$ and $\beta_k\leq 1$. Therefore, we achieve 
\begin{equation*}
    \dist\biggl(0,\partial^L\Phi(z^{k+1})\biggr) \leq \kappa\|z^k - z^{k+1}\|,
\end{equation*}
where we have used the Cauchy-Schwarz inequality and $\kappa = (LL_\phi + L_f+ L_h+(1+\delta)L/2(M + L_\phi))\sqrt{2}$.

($\mathbf H3$) \emph{Continuity condition}. By the boundedness of $\{x^k\}$ and the part (b) of Theorem \ref{theo1}, there exists a sub-sequence $\{z^{k_j}\} = \{(x^{k_j},x^{k_j-1})\}$ of $\{z^k\}$ that converges to a point $z^* = (x^*,x^*)$. We need to prove that $\Phi(z^{k_j})\to \Phi(z^*)$ as $j\to+\infty$. It follows from (b) of Theorem \ref{theo1} and $0\leq \beta_k<1$ that $\lim x^{k_j-1} = \lim x^{k_j-2} = \lim y^{k_j-1} = x^*$. It follows from the definition of $x^{k_j}$ that 
\begin{equation}\label{equ8121}
    \begin{split}
        G(x^{k_j}) \leq G(x^*) - \langle L\phi(y^{k_j-1}) - \nabla f(y^{k_j-1})
      + \xi^{k_j-1} , x^* - x^{k_j}\rangle
    \end{split}
\end{equation}
Thanks to the fact that $\lim x^{k_j-1} = \lim x^{k_j-2} = \lim y^{k_j-1} = x^*$, the continuity of $\nabla f, \nabla \phi$, and the boundedness of $\{\xi^{k_j-1}\}$, taking to the limit in \eqref{equ8121} gives us that
\begin{equation}
    \limsup_{j\to+\infty}G(x^{k_j}) \leq G(x^*).
\end{equation}
Therefore, we have
\begin{equation}
    \begin{split}
        \limsup_{j\to+\infty}F(x^{k_j}) &=\limsup\limits_{j\rightarrow\infty}[G(x^{k_j}) - H(x^{k_j})]\\
& \leq \limsup\limits_{j\to+\infty}G(x^{k_j}) - \liminf\limits_{j\to+\infty}H(x^{k_j})\\
& \leq G(x^*) - \liminf\limits_{j\to+\infty}H(x^{k_j})\\
& \leq G(x^*) - H(x^*) = F(x^*).
    \end{split}
\end{equation}
On the other hand, from the lower semicontinuity of $F$, we obtain 
\begin{equation}
    \lim\inf_{j\to+\infty}F(x^{k_j}) \geq F(x^*).
\end{equation}
Hence, by the uniqueness of limit and the definition of $\Phi$, we have $\lim\limits_{j\to+\infty}\Phi(z^{k_j}) = \lim\limits_{j\to+\infty}F(x^{k_j}) = F(x^*) = \Phi(z^*)$.
The result now follows from the same arguments of the proof for \cite[Theorem 2.9]{attcon}.

\end{proof}

Moreover, if the function $\psi$ appearing in the KL inequality takes the form $\psi(s) = cs^{1-\theta}$ with $\theta\in [0,1)$ and $c>0$, we can derive the convergence rates for the both sequences $\{x^k\}$ and $\{F(x^k)\}$.
\begin{theorem}
Let $\{x^k\}$ be a sequence generated by DCAe. Suppose that the assumptions \ref{assump1} and \ref{assump2}  are satisfied. If $\{x^k\}$ is bounded and the function $\psi$ appearing in the KL inequality takes the form $\psi(s) = cs^{1-\theta}$ with $\theta\in [0,1)$ and $c>0$, the following statements hold. 
\begin{itemize}
    \item[a)] If $\theta = 0$, the sequences $\{x^k\}$ and $\{F(x^k)\}$ converge in a finite number of steps to $x^*$ and $F^*$, respectively.

\item[b)] If $\theta \in (0,1/2]$, the sequences $\{x^k\}$ and $\{F(x^k)\}$ converge linearly to $x^*$ and $F^*$, respectively.

\item[c)] If $\theta \in (1/2,1)$, there exist positive constants $\delta_1$, $\delta_2$, and $N$ such that
$\|x^k - x^*\| \leq \delta_1 k^{-\frac{1-\theta}{2\theta-1}}$  and $F(x^k) - F^* \leq \delta_2 k^{-\frac{1}{2\theta-1}}$ 
for all $k \geq N$.
\end{itemize}
\end{theorem}
Since this theorem can be proved by using the same techniques as in the proofs of \cite[Theorem 2]{hedon} and \cite[Theorem 3.4]{Frankel2015}, we omit the detail of the proof. 

\section{Numerical Experiments}\label{numerical}
We evaluate DCAe with regard to the nonconvex negative matrix completion problem \eqref{MF}, namely
\begin{equation}\label{MF_exp}
    \min_{\mbfU\in\mathbb{R}^{m\times t},\mbfV\in\mathbb{R}^{t\times n}} \biggl\{ F(\mbfU,\mbfV) := f(\mbfU,\mbfV) + \chi_{\mathbb{R}^{m\times t}_+\times \mathbb{R}^{t\times n}_+}(\mbfU,\mbfV) + r(\mbfU,\mbfV)\biggr\},
\end{equation}
where $f(\mbfU,\mbfV): = \frac{1}{2}\|\mathcal P(\mbfA - \mbfU\mbfV)\|_F^2$ is the data-fitting term, $r$ is a regularization term, and $\chi_{\mathbb{R}^{m\times t}_+\times \mathbb{R}^{t\times n}_+}$ is the indicator function. Here, $\mathcal P(\mbfZ)_{ij} = \mbfZ_{ij}$ if $\mbfA_{ij}$ is observed and $0$ otherwise. We consider the exponential concave regularization \cite{Bradley1998}:
    \begin{equation}\label{exp}
        \begin{split}
        r(\mbfU,\mbfV) = \lambda\biggl(\sum_{ij}\biggl(1-\exp(-\theta|u_{ij}|)\biggr) + \sum_{ij}\biggl(1-\exp(-\theta|v_{ij}|)\biggr) \biggr),
    \end{split}
    \end{equation}
    where $u_{ij}$ are elements of $\mbfU$. We choose a function $\phi$ given by
\begin{equation*}
    \phi(\mbfU,\mbfV) = c_1\biggl(\frac{\|\mbfU\|^2_F +\|\mbfV\|^2_F}{2}\biggr)^2 + c_2\biggl(\frac{\|\mbfU\|^2_F +\|\mbfV\|^2_F}{2}\biggr),
\end{equation*}
where $c_1 = 3$ and $c_2 = \|P_O(\mbfA)\|_F$. 
In \cite{mukbey}, the authors showed that $L\phi - f$ and $L\phi + f$ are convex for all $L\geq 1$. The exponential concave regularization can be expressed as a DC function $r = \lambda\theta\biggl(\|\mbfU\|_1 + \|\mbfV\|_1\biggr) - h$, where  $h$ is defined by
\begin{equation*}
    \begin{split}
        h(\mbfU,\mbfV) =  \lambda\theta\biggl(\|\mbfU\|_1 + \|\mbfV\|_1\biggr) - r(\mbfU,\mbfV).
    \end{split}
\end{equation*}
Therefore, the problem \eqref{MF_exp} takes the form of \eqref{model} with $g = \chi_{\mathbb{R}^{m\times t}_+\times \mathbb{R}^{t\times n}_+}(\mbfU,\mbfV) + \lambda\theta\biggl(\|\mbfU\|_1 + \|\mbfV\|_1\biggr)$. 
The DCAe algorithm iteratively computes $\xi^k\in \partial h(\mbfU^k,\mbfV^k)$ by
\begin{equation*}
\begin{split}
     \xi^k_{u_{ij}} = \lambda\theta(1-\exp(-\theta|u_{ij}^k|))\text{sign}(u_{ij}^k)\\
     \xi^k_{v_{ij}} = \lambda\theta(1-\exp(-\theta|v_{ij}^k|))\text{sign}(v_{ij}^k),
\end{split}
\end{equation*}
and solves the following convex nonsmooth problem:
\begin{equation}\label{subMF1}
\begin{split}
    \min_{\mbfU\in\mathbb{R}^{m\times t}_+,\mbfV\in\mathbb{R}^{t\times n}_+}\biggl\{\lambda/L(\|\mbfU\|_1 + \|\mbfV\|_1) + \langle\mbfP^k, \mbfU\rangle + \langle\mbfQ^k, \mbfU\rangle\\ +
    c_1\biggl(\frac{\|\mbfU\|^2_F +\|\mbfV\|^2_F}{2}\biggr)^2 + c_2\biggl(\frac{\|\mbfU\|^2_F +\|\mbfV\|^2_F}{2}\biggr)\biggr\},
\end{split}
\end{equation}
where $\mbfP^k$ and $\mbfQ^k$ are defined by
\begin{equation*}
    \begin{split}
        \begin{split}
        \mbfP^k = \nabla_U f(\mbfY^k,\mbfZ^k)/L - \xi^k_U/L- \nabla_U\phi(\mbfY^k,\mbfZ^k)\\
        \mbfQ^k  = \nabla_V f(\mbfY^k,\mbfZ^k)/L - \xi^k_V/L - \nabla_U\phi(\mbfY^k,\mbfZ^k),
    \end{split}
    \end{split}
\end{equation*}
with $\mbfY^k = \mbfU^k +\beta_k(\mbfU^k - \mbfU^{k-1})$ and  $\mbfZ^k = \mbfV^k +\beta_k(\mbfV^k - \mbfV^{k-1})$. 
The solution to the convex sub-problem \eqref{subMF1} is given by $\mbfU^{k+1} = \tau^*[-\mbfP^k - \lambda/L]_+$ and $\mbfV^{k+1} = \tau^*[-\mbfQ^k - \lambda/L]_+$, where $([-\mbfP -  \lambda/L]_+)_{ij} = \max(0,-p_{ij} - \lambda/L)$ and $\tau^*$ is the unique positive real root of
\begin{equation*}
    c_1\biggl(\|[-\mbfP^k - \lambda/L]_+\|^2_F + \|[-\mbfQ^k - \lambda/L]_+\|^2_F\biggr)\tau^3 + c_2\tau - 1 = 0.
\end{equation*}
In our experiments, we choose $L = l = 1$ and $\delta = 0.9999$. At each iteration, we initialize $\beta_k = \frac{\mu_k-1}{\mu_k}$, and while the inequality \eqref{equ:beta} does not hold, we decrease $\beta_k$ by $\beta_k = 0.9\beta_k$, where $\mu_k = \frac{1}{2}(1+\sqrt{1+4\mu_{k-1}^2})$ and $\mu_0=1$. 

To verify the effect of the extrapolation term, we compare our DCAe algorithm with its non-extrapolated version, which is DCA, and the inertial DCA algorithm (iDCA) \cite{aniwel}.
\begin{table}[tbh!]
\centering
\caption{Dataset Descriptions. The number of users, items, and ratings used in each dataset}\label{dataset}
\begin{tabular}{@{}cllll@{}}
\hline
\multicolumn{2}{l}{Dataset}       & \multicolumn{1}{l}{\#users} & \multicolumn{1}{l}{\#items} & \multicolumn{1}{l}{\#ratings} \\ 
\hline
\multirow{2}{*}{MovieLens}  & 1M  & 6,040                       & 3,449                       & 999,714                       \\
                            & 10M & 69,878                      & 10,677                      & 10,000,054                    \\
\multicolumn{1}{l}{Netflix} &     & 480,189                     & 17,770                      & 100,480,507                   \\ \hline
\end{tabular}
\end{table}
All the experiments were conducted on a PC 2.3 GHz Intel Core i5
of 8GB RAM. The codes were written in MATLAB and are available from \url{https://github.com/nhatpd/DCAe}.

We set $\lambda = 0.1$ and $\theta = 5$ as proposed in \cite{Bradley1998}. We note that we do not optimize numerical results by tweaking the parameters as this is beyond the scope of this work. 
It is important noting that we evaluate the algorithms on the same models. 
We carried out the experiments on the two most widely used datasets in the field of recommendation systems, MovieLens and Netflix, which contain ratings of different users. 
The characteristics of the datasets are given in Table~\ref{dataset}.
We respectively choose $t = 5, 8$, and $13$ for MovieLens 1M, 10M, and Netflix data set. We  randomly used 70\% of the observed ratings for training and the rest for testing. The process was repeated twenty times. We run each algorithm 20, 200, and 3600 seconds for MovieLens 1M, 10M, and Netflix data set, respectively. 
We are interested in the root mean squared error on the test set: $RMSE = \sqrt{\|\mathcal P_T(A - UV)\|^2/N_T}$, where $\mathcal P_T(Z)_{ij} = Z_{ij}$ if $A_{ij}$ belongs to the test set and $0$ otherwise, $N_T$ is the number of ratings in the test set.
 We plotted the curves of the average value of RMSE and the objective function value (log scale) versus training time in Figure \ref{fig:MCP} and report the average and the standard deviation of RMSE and the objective function value in Table~\ref{results}. 

\begin{figure*}[ht]
\begin{center}
\begin{tabular}{cc}
\includegraphics[width=0.49\textwidth]{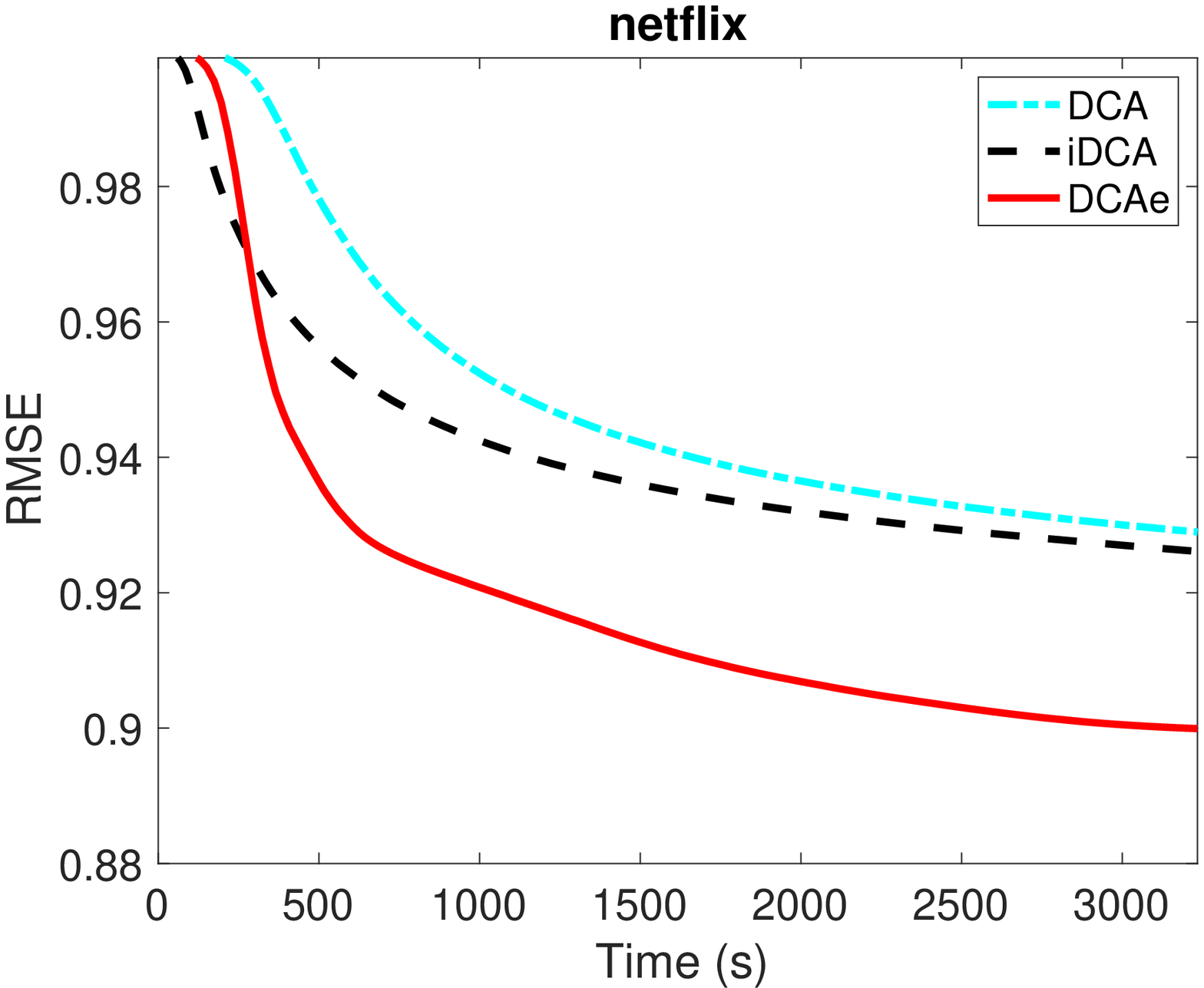}  & 
\includegraphics[width=0.49\textwidth]{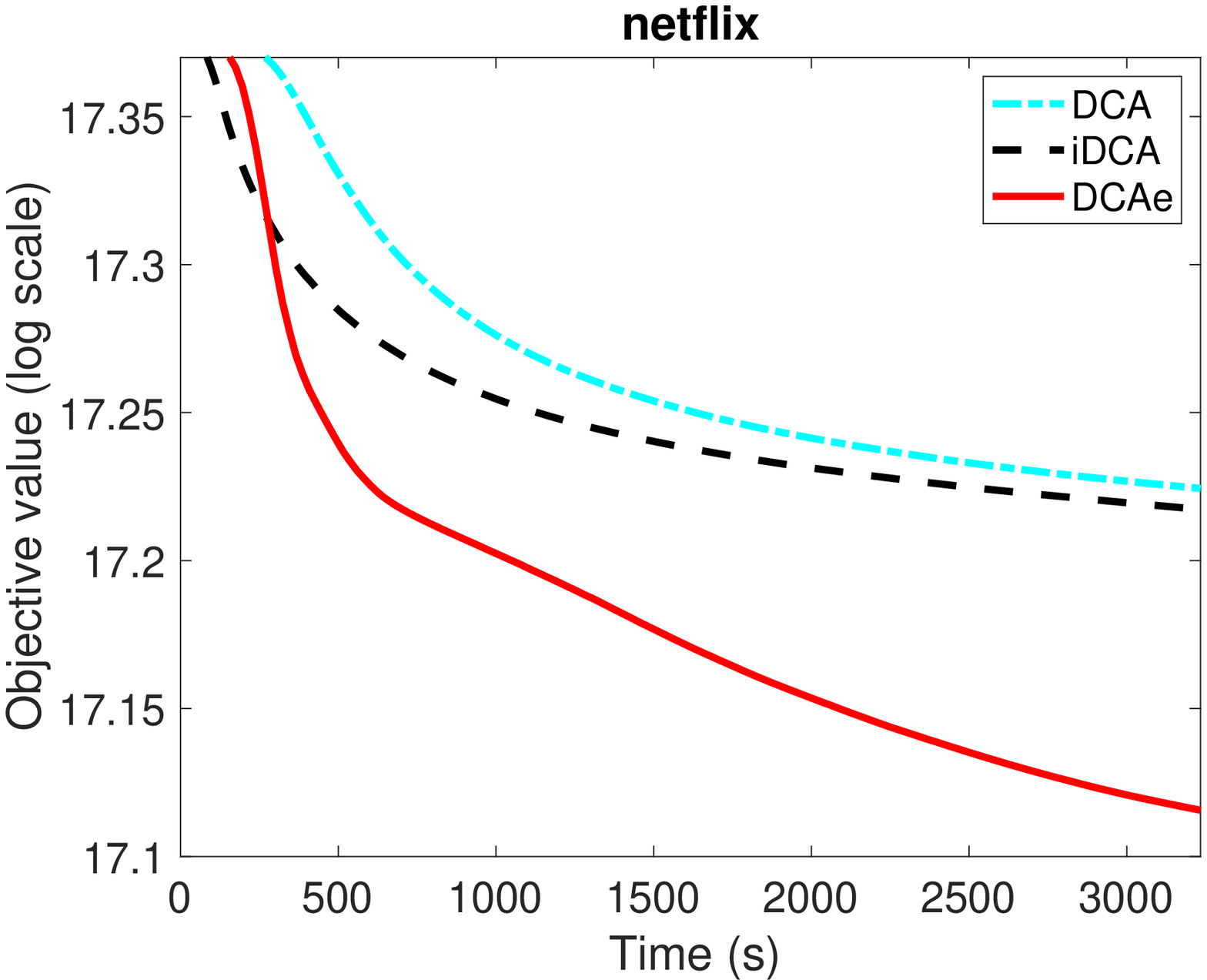} \\
\includegraphics[width=0.49\textwidth]{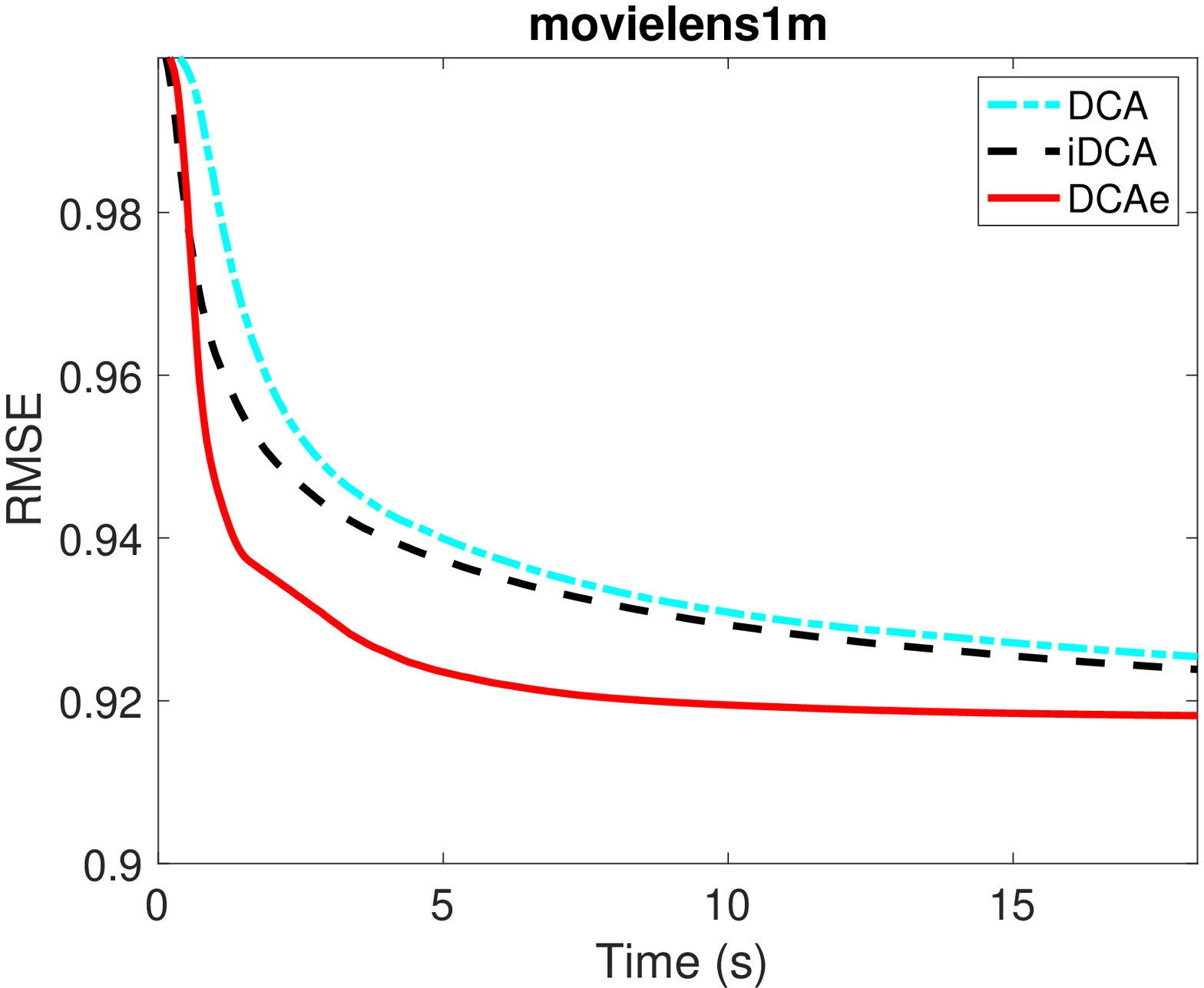}  & 
\includegraphics[width=0.49\textwidth]{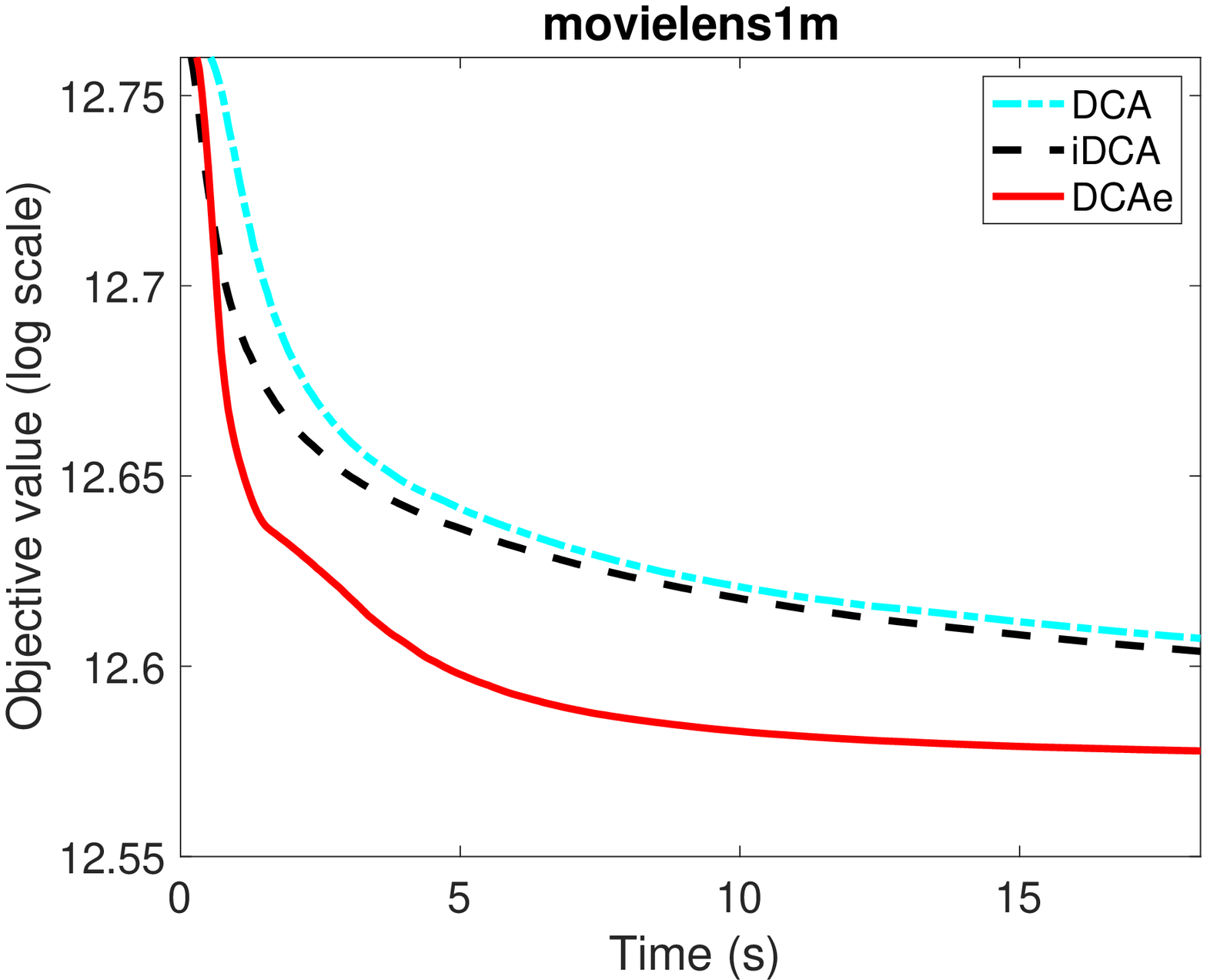} \\
\includegraphics[width=0.49\textwidth]{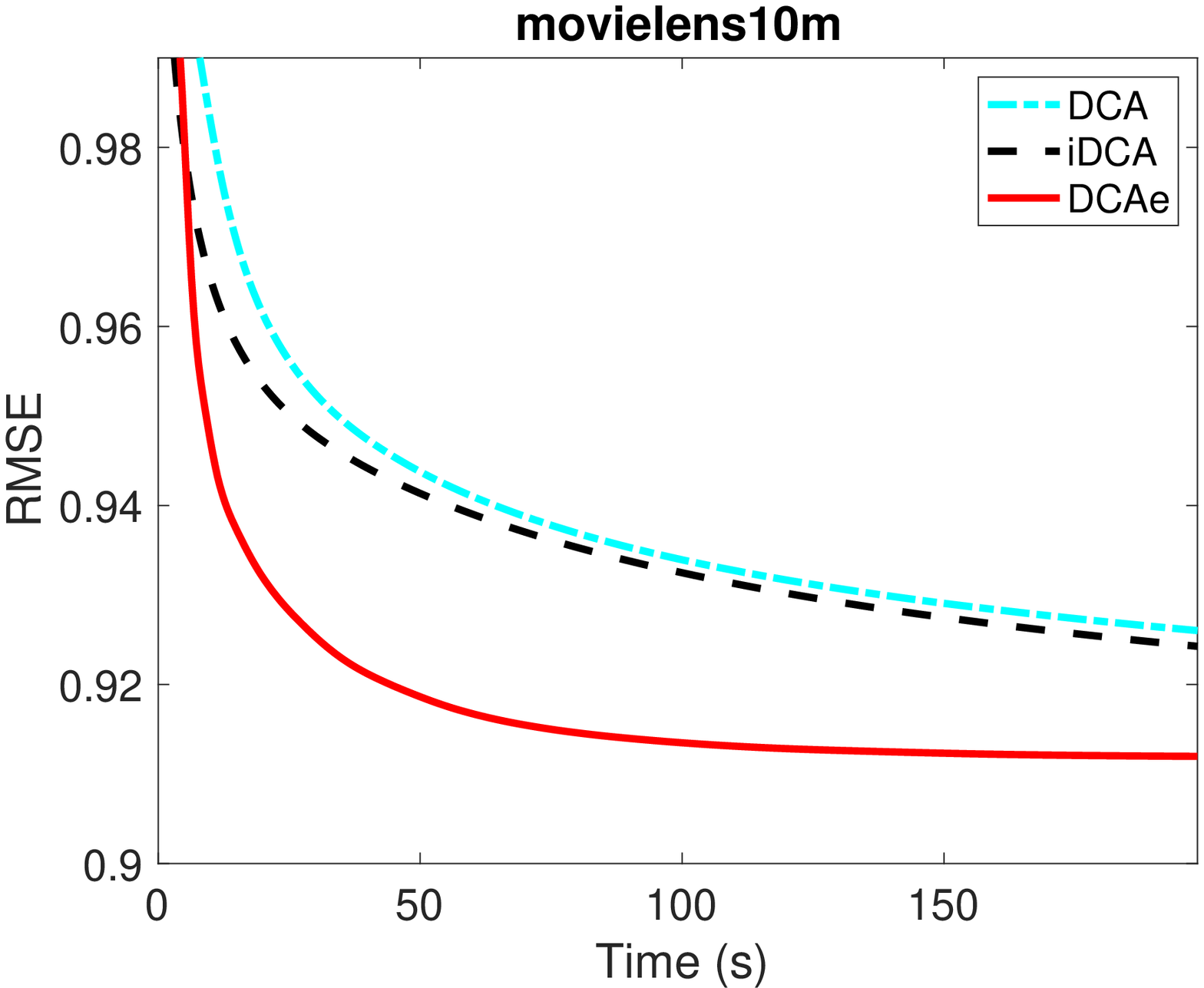}  & 
\includegraphics[width=0.49\textwidth]{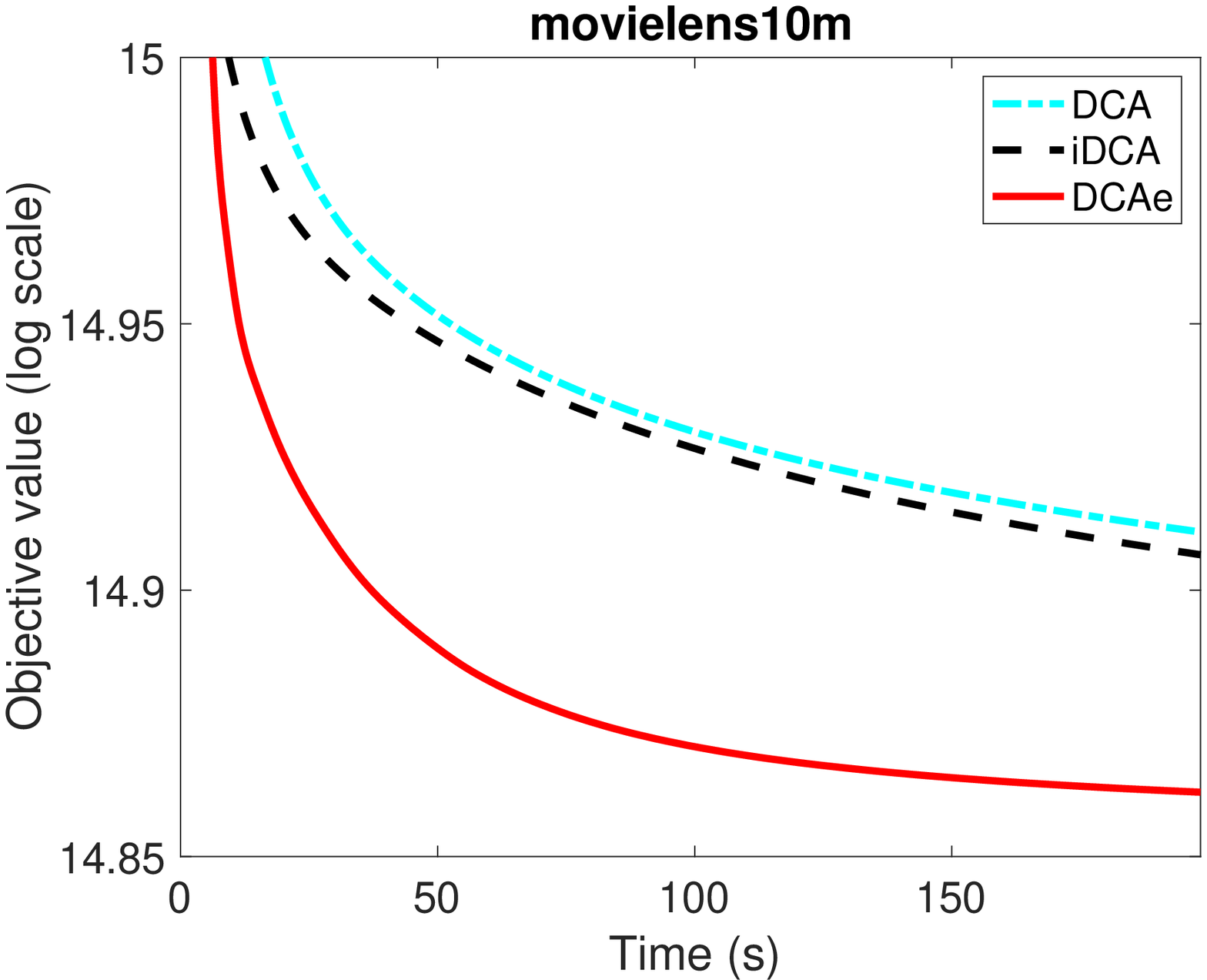}
\end{tabular}
\caption{Evolution of the average value of the RMSE on the test set and the objective function value with respect to time 
\label{fig:MCP}} 
\end{center}
\vspace{-0.2in}
\end{figure*}

\begin{table}[tb]
\centering
\caption{Bold values indicate the best results for each dataset.}\label{results}
\begin{tabular}{@{}clll@{}}
\hline
\multirow{2}{*}{Dataset}      & \multirow{2}{*}{Method} & RMSE & Objective value \\
& & mean $\pm$ std & (mean $\pm$ std)$\times 10^{-5}$ \\ 
 \hline
  \multirow{3}{*}{MovieLens 1M}   & DCA &  0.9247 $\pm$   0.0007  & 2.9814 $\pm$   0.0041 \\
  & iDCA &  0.9231  $\pm$  0.0003  & 2.9715 $\pm$   0.0033\\
  
 & DCAe  & \textbf{0.9181}  $\pm$  0.0007 & \textbf{2.8989} $\pm$   0.0030\\
  \hline
  \multirow{3}{*}{MovieLens 10M} & DCA  & 0.9259    $\pm$ 0.0005  &  29.8973  $\pm$  0.0342\\
  & iDCA  & 0.9242  $\pm$  0.0006  & 29.7692 $\pm$   0.0418\\
  
  & DCAe & \textbf{0.9120}  $\pm$  0.0003 & \textbf{28.4760}   $\pm$ 0.0197\\
  
\hline
  \multirow{3}{*}{Netflix} & DCA & 0.9275 $\pm$   0.0006  &  301.2471  $\pm$  0.4495\\
  & iDCA  & 0.9247  $\pm$  0.0009 & 299.0971 $\pm$   0.7649\\
  
  & DCAe & \textbf{0.8996} $\pm$ 0.0004 & \textbf{269.2313}  $\pm$  1.3008\\
  
\hline
\end{tabular}
\vspace{-0.1in}
\end{table}

We observe that DCAe converges the fastest on all the data sets, providing a significant acceleration of DCA and iDCA. DCAe obtains not only the best final objective function values but also the best RMSE on the test set. 
This illustrates the usefulness of the extrapolation term in DCA.  

\section{Conclusions}\label{conclu}
We developed the DCAe algorithm for minimizing the sum of a nonconvex differentiable function and a DC function, a novel variant of DCA with extrapolation that subsumes the pDCAe algorithm in \cite{wenapr} as a special case. We studied the sub-sequential and global convergence of DCAe to a critical point under mild and stronger conditions, respectively. To evaluate the performance of our proposed algorithm, we applied DCAe to the nonnegative matrix completion problem . The numerical results strongly confirmed the advantage of DCAe to DCA and iDCA.

\bibliography{DCAe}

\bibliographystyle{abbrv}

\end{document}